\documentclass[12pt]{article}
\usepackage[misc]{ifsym}
\usepackage[T1]{fontenc}
\usepackage[utf8]{inputenc}
\usepackage{authblk}
\usepackage{indentfirst}
\usepackage[top=1in, bottom=1in, left=1in, right=1in]{geometry}

\usepackage{amsmath}
\usepackage{enumerate}
\usepackage{amsthm}
\usepackage{abstract}

\usepackage{bm}
\usepackage{graphics}
\usepackage{graphicx}
\usepackage{subfigure}
\usepackage{caption}
\usepackage{amsthm,amsmath,amssymb}
\usepackage{float}
\usepackage[section]{placeins}
\usepackage{multirow}
\usepackage{amssymb}
\usepackage{IEEEtrantools}
\usepackage{cite}
\usepackage{pifont}
\usepackage{mathrsfs}
\pdfoutput=1

\title{A Brualdi–Hoffman–Tur\'{a}n problem for friendship graph}
\date{}
\author{Fan Chen$^{1,\, 2}$, Xiying Yuan$^{1,\, 2}$ \thanks{Corresponding author.\\ \indent This work is supported by the National Natural Science Foundation of China (Nos. 11871040,
12271337, 12371347) \\ \indent Email address: xiyingyuan@shu.edu.cn (Xiying Yuan) chenfan@shu.edu.cn (Fan Chen)}}
\affil{{\footnotesize\emph{1. Department of Mathematics, Shanghai University, Shanghai 200444, P.R. China}}\\
{\footnotesize\emph{2. Newtouch Center for Mathematics of Shanghai University, Shanghai 200444, P.R. China}}}

\begin{document}
\newtheorem{theorem}{Theorem}[section]
\newtheorem{assumption}[theorem]{Assumptio}
\newtheorem{corollary}[theorem]{Corollary}
\newtheorem{proposition}[theorem]{Proposition}
\newtheorem{lemma}[theorem]{Lemma}
\newtheorem{definition}[theorem]{Definition}
\newtheorem{remark}[theorem]{Remark}
\newtheorem{problem}[theorem]{Problem}
\newtheorem{claim}{Claim}
\newtheorem{conjecture}[theorem]{Conjecture}
\newtheorem{fact}{Fact}

\maketitle
\noindent\rule[0pt]{17.5cm}{0.09em}

\noindent{\bf Abstract}

A graph is said to be $H$-free if it does not contain $H$ as a subgraph. Brualdi–Hoffman–Tur\'{a}n type problem is to determine the maximum spectral radius of an $H$-free graph $G$ with give size $m$. The $F_k$ is the graph consisting of $k$ triangles that intersect in exactly one common vertex, which is known as the friendship graph. In this paper, we resolve a conjecture (the Brualdi–Hoffman–Tur\'{a}n-type problem for $F_k$) of Li, Lu and Peng [Discrete Math. 346 (2023) 113680] by using  the $k$-core technique presented in Li, Zhai and Shu [European J. Combin, 120 (2024) 103966].

\noindent



\noindent\rule[0pt]{17.5cm}{0.05em}

\section{Introduction}
Let $G$ be a simple graph with vertex set $V(G)$ and edge set $E(G)$. We usually write $n$ and $m$ for the number of vertices and edges respectively. For a vertex $u\in V(G)$, let $N_G(u)$ be the neighborhood of $u$ and $N_G[u] = N_G(u)\cup\{u\}$. For simplicity, we use $N_S(u)$ to denote $N_G(u)\cap S$. For a subset $A\subseteq V(G)$, we write $e(A)$ for the number of edges with two end vertices in $A$. For two disjoint sets $A$, $B$, we write $e(A, B)$ for the number of edges between $A$ and $B$. Let $P_n$, $C_n$, and $K_n$ be the path, the cycle and the complete graph on $n$ vertices, respectively.

For two disjoint graphs $G$ and $H$, let $G\cup H$ be the graph with vertex set $V(G)\cup V(H)$ and edge set $E(G)\cup E(H)$. The join of two disjoint graphs $G$ and $H$, denoted by $G\vee H$, is obtained from $G\cup H$ by adding all possible edges between $G$ and $H$.

A graph $G$ is called $H$-free if it does not contain a copy of $H$ as a subgraph. The classical Tur\'{a}n problem is to determine the maximum number of edges, $ex(n, H)$, of an $n$ vertex $H$-free
graph, where the value $ex(n, H)$ is called the Tur\'{a}n number. Denote $Ex(n, H)$ for
the set of all $n$ vertex $H$-free graphs with maximum number of edges. Let $F_k$ denote the graph consisting of $k$ triangles that intersect in exactly one common vertex and this graph is known as the friendship graph. In 1995, Erd\H{o}s, F\"{u}redi, Gould and Gunderson proved the following result.

\begin{theorem}\cite{EFGG}
For every $k \geqslant 1$ and $n \geqslant 50k^2$,
$$ex(n,F_k)=\bigg\lfloor\frac{n^2}{4}\bigg\rfloor+\left\{
     \begin{array}{ll}
       k^2-k, & \hbox{if $k$ is odd;} \\
       k^2-\frac{3}{2}k, & \hbox{if $k$ is even.}
     \end{array}
 \right.$$
\end{theorem}

Let $A(G)$ be the adjacency matrix of graph $G$. Since $A(G)$ is real symmetric, its eigenvalues are real and hence can be ordered as $\lambda_1(G)\geqslant\cdots\geqslant\lambda_n(G)$, where $\lambda_1(G)$ is called the spectral radius of G and also denoted by $\rho(G)$. In \cite{N1}, Nikiforov  proposed a spectral Tur\'{a}n  type problem is to determine the maximum spectral radius of an $H$-free graph of order $n$, which is also called the Brualdi-Solheid-Tur\'{a}n type problem. The study of the Brualdi-Solheid-Tur\'{a}n type problem has a long history, such as $K_{r+1}$ \cite{N2, W}, $K_{s,t}$ \cite{BG, N2, N3}, $P_k$ \cite{N4}, $C_{2k+1}$ \cite{N5}, $C_{2k+2}$ \cite{CDT}. For more details, we refer to read a survey \cite{LN}. The Brualdi-Solheid-Tur\'{a}n type problem for $F_k$ was characterized by Cioab\v{a}, Feng, Tait and Zhang in the following.

\begin{theorem}\cite{CFTZ}
Let $k\geqslant 2$ and G be an $F_k$-free graph on $n$ vertices. For sufficiently large $n$, if $G$ has
the maximal spectral radius, then $G\in Ex(n, F_k)$.
\end{theorem}

Another interesting spectral Tur\'{a}n type problem asks what is the maximum spectral radius of an $H$-free graph with given size $m$, which is also known as Brualdi-Hoffman-Tur\'{a}n type problem (see \cite{BH}). In \cite{N6}, Nosal showed that for any triangle-free graph $G$ with $m$ edges, we have $\rho(G)
\leqslant \sqrt{m}$. Subsequently, Nikiforov in \cite{N7, N8} extended Nosal's result to $K_{r+1}$.
\begin{theorem} \cite{N7, N8}
If $G$ is $K_{r+1}$-free with $m$ edges, then
$$\rho(G)\leqslant \sqrt{2m(1-\frac{1}{r})}.$$
with equality if and only if $G$ is a complete bipartite graph for $r=2$, and $G$ is a complete regular $r$-partite graph for $r \geqslant 3$.
\end{theorem}
Furthermore, Bollob\'{a}s and Nikiforov in \cite{BN} conjectured that if $G$ is a $K_{r+1}$-free graph of order at least $r+1$ with $m$ edges, then
$$\lambda_1^2(G)+\lambda_2^2(G)\leqslant 2m(1-\frac{1}{r}).$$
Ando and Lin proved the conjecture holds for $r$-partite graphs in \cite{AL}. The case $r=2$ was confirmed by Lin, Ning and Wu in \cite{LNW}. Recently, Zhang in \cite{Z} proved the conjecture holds for regular graphs.

A series of problems mentioned above have stimulated a great deal of interest in the study of Brualdi-Hoffman-Tur\'{a}n type problem, such as $C_4$ \cite{N9}, $K_{2,r+1}$ \cite{ZLS}, $C_5$ or $C_6$ \cite{MLH, SLW}, $C_7$ or $C_7^+$ \cite{LLL}, $C_5^+$ or $C_6^+$ \cite{SLW, FY, LW}, $C_{2k+1}^+$ or $C_{2k+2}^+$ \cite{LZS} $K_1\vee P_4$ \cite{ZW, YLP}, $K_1\vee P_6$ \cite{ZW2}, where $C_k^+$ is a graph on $k$ vertices obtained from $C_k$ by adding a chord between two vertices with distance two. In \cite{LZS}, Li, Zhai and Shu used a crucial "$k$-core" technique which was defined in \cite{PSW}.

Recall that $F_k$ is the friendship graph. Recently, in \cite{LLP}, Li, Lu and Peng proposed a conjecture about $F_k$, and they proved the case $k=2$.

\begin{conjecture}\cite{LLP}
Let $k\geqslant 2$ and $m$ large enough. If $G$ is $F_k$-free with given size $m$, then
$$\rho(G)\leqslant \frac{k-1+\sqrt{4m-k^2-1}}{2}$$
with equality if and only if $G\cong K_k\vee(\frac{m}{k}-\frac{k-1}{2})K_1$.
\end{conjecture}
In \cite{YLP}, Yu, Li and Peng confirmed the conjecture in $k=3$. In this paper, we completely solve the conjecture in the following theorem.
\begin{theorem}\label{th1}
Let $k\geqslant 2$ and $m\geqslant 4k^4$. If $G$ is $F_k$-free with given size $m$, then
$$\rho(G)\leqslant \frac{k-1+\sqrt{4m-k^2-1}}{2}$$
with equality if and only if $G\cong K_k\vee(\frac{m}{k}-\frac{k-1}{2})K_1$.
\end{theorem}

\section{Preliminaries}
Before proceeding, we first present some important results and definitions that are useful for our proof. Let $S_{n,k} = K_k\vee (n-k)K_1$.
\begin{lemma}\cite{EG}\label{l1}
Let $k\geqslant2$ be an integer. Then
$$ex(n,kK_{2})=\left\{
     \begin{array}{ll}
       (k-1)n-\frac{k(k-1)}{2}, & \hbox{if $n\geq(5k-2)/2$;} \\
       \binom{2k-1}{2}, & \hbox{if $2k-1\leq n<(5k-2)/2$.}
     \end{array}
 \right.$$
If\ $n > (5k-2)/2$, then $Ex(n,kK_2)$ = $\{S_{n,k-1}\}$; if\  $n = (5k-2)/2$, then $Ex(n,kK_2) = \ \ \{K_{2k-1}\cup(n-2k+1)K_1,S_{n,k-1}\}$; if\ $2k-1 \leq n < (5k-2)/2$, then $Ex(n,kK_2) = \{K_{2k-1}\cup(n-2k+1)K_1\}$.
\end{lemma}

The following lemma is known as Erd\H{o}s–Gallai theorem.
\begin{lemma}\cite{EG}\label{l2}
Let $k\geqslant 2$ and $G$ be an $n$ vertex $\{C_{\ell}:\ell>k\}$-free graph. Then $e(G)\leqslant \frac{1}{2}k(n-1)$, equality holds if and only if $n=t(k-1)+1$ and $G\cong K_1\vee (tK_{k-1})$.
\end{lemma}

For convenience, we will use $\rho$ and $\bm{x}$ to denote the spectral radius and the positive vector of $G$. Suppose $x_{u^*}=\max\{x_v:v\in V(G)\}=1$. Furthermore, we denote $U=N_G(u^*)$, $W=V(G)\setminus N_G[u^*]$ and $d_U(u) = |N_U(u)|$ for a vertex $u\in V(G)$. Since
\begin{equation}\label{e2}
\rho^2=\rho^2x_{u^*}=d_G(u^*)+\sum_{u\in U}d_U(u)x_u+\sum_{w\in W}d_U(w)x_w,
\end{equation}
which together with $\rho=\rho x_{u^*}=\sum_{u\in U}x_u$ yields
\begin{equation}\label{e3}
\rho^2-(k-1)\rho=d_G(u^*)+\sum_{u\in U}(d_U(u)-k+1)x_u+\sum_{w\in W}d_U(w)x_w.
\end{equation}

We shall now present the definition of the $k$-core, which will form the basis of our proof. A $k$-core of a graph G is defined as the largest induced subgraph of $G$ such that its minimum degree is at least $k$. A $k$-core is obtained from $G$ by removing vertices of degree at most $k-1$ until the minimum degree is at least $k$. In \cite{PSW}, we can know that the $k$-core is well-defined, which is not affected by the order of vertex deletion.

We denote by $V^c$ the vertex set of the $(k-1)$-core of $G[V]$. For an arbitrary vertex set $V$, we define $e(V)=e(G[V])$ and
$$\eta_1(V)=\sum_{u\in V}(d_V(u)-k+1)x_u-e(V).$$
If $V=\emptyset$, then we define $\eta_1(V)=0$.
\begin{lemma}\cite{LZS}\label{vc}
For every subset $V$ of $U$, we have $\eta_1(V) \leqslant \eta_1(V^c)$, where the equality holds if and only if $V=V^c$.
\end{lemma}
\section{Proof of Theorem \ref{th1}}
Let $G$ be the maximum spectral radius of $F_k$-free graphs with $m$ edges and $\rho$ is the spectral radius corresponding with positive eigenvector $\bm{x}$. Suppose $x_{u^*}=\max\{x_v:v\in V(G)\}=1$. Noting that $\rho\geqslant\rho(K_k\vee(\frac{m}{k}-\frac{k-1}{2})K_1)=\frac{k-1+\sqrt{4m-k^2-1}}{2}$, we have
\begin{equation}\label{e1}
\rho^2-(k-1)\rho\geqslant m-\frac{1}{2}k(k-1),
\end{equation}
Otherwise, $\rho^2-(k-1)\rho< m-\frac{1}{2}k(k-1)$, which is a contradiction to the assumption of $G$.
Combining inequalities (\ref{e3}) and (\ref{e1}), we have
\begin{align}
&m-\frac{1}{2}k(k-1)\notag\\
=&d_G(u^*)+e(U)+e(W)+e(U,W)-\frac{1}{2}k(k-1)\notag\\
\leqslant&d_G(u^*)+\sum_{u\in U}(d_U(u)-k+1)x_u+\sum_{w\in W}d_U(w)x_w\notag\\
\leqslant&d_G(u^*)+\sum_{u\in U}(d_U(u)-k+1)x_u+e(U,W).
\end{align}
Furthermore, we obtain
\begin{equation}\label{e5}
\eta_1(U)=\sum_{u\in U}(d_U(u)-k+1)x_u-e(U)\geqslant e(W)-\frac{1}{2}k(k-1)\geqslant -\frac{1}{2}k(k-1).
\end{equation}
If $\eta_1(U)=-\frac{1}{2}k(k-1)$, then $W=\emptyset$ or $e(W)=0$ and $x_w=1$ for each $w\in W$.

Let $\mathcal{H}$ be the family of connected components in $G[U^c]$ and $|\mathcal{H}|$ be the number of
members in $\mathcal{H}$. By the definition of $(k-1)$-core, for each $H \in \mathcal{H}$, we have $\delta(H) \geqslant k-1$, and then $|H|\geqslant k$. Let $\mathcal{H}_1=\{H\in \mathcal{H}:|H|>\frac{5k-2}{2}\}$. Hence $\mathcal{H}\setminus \mathcal{H}_1=\{H\in \mathcal{H}:k\leqslant |H|\leqslant\frac{5k-2}{2}\}$.

\begin{lemma}\label{h1}
For each $H\in \mathcal{H}_1$, we have $\eta_1(V(H))\leqslant -\frac{1}{2}k(k-1)$, with equality holds if and only if $x_u=1$ for each $u\in V(H)$ with $d_H(u)\geqslant k$ and $H\cong S_{|H|,k-1}$.
\end{lemma}
\begin{proof}
Since $G$ is $F_k$-free, we have $G[U]$ is $kK_2$-free, and then $H$ is also $kK_2$-free. Combining the assumption $|H|>\frac{5k-2}{2}$ and Lemma \ref{l1}, we have
$$e(H)\leqslant (k-1)|H|-\frac{1}{2}k(k-1),$$
with equality holds if and only if $H\cong S_{|H|,k-1}$. Furthermore, we have
\begin{align*}
\eta_1(V(H))&=\sum_{u\in V(H)}(d_H(u)-k+1)x_u-e(H)\\
&\leqslant \sum_{u\in V(H)}(d_H(u)-k+1)-e(H)\\
&=e(H)-(k-1)|H|\leqslant-\frac{1}{2}k(k-1),
\end{align*}
with equality holds if and only if $x_u=1$ for each $u\in V(H)$ with $d_H(u)\geqslant k$ and $H\cong S_{|H|,k-1}$.
\end{proof}

Since $G[U]$ is $kK_2$-free, we have $G[U]$ does not contain any cycle of length large than $2k-1$, and then $H$ does not contain any cycle of length large than $2k-1$ for any $H\in \mathcal{H}$. Let $\mathcal{H}_2=\{H\in \mathcal{H}:k\leqslant |H|\leqslant\frac{5k-2}{2} \text{and}\ H\ \text{contains}\ C_{2k-1}\ \text{as a subgraph}\}$ and $\mathcal{H}_3=\mathcal{H}\setminus (\mathcal{H}_1\cup \mathcal{H}_2)$. Hence, for any graph $H\in \mathcal{H}_3$, $H$ does not contain $C_{2k-1}$ as a subgraph. Furthermore, for any graph $H\in \mathcal{H}_3$, $H$ does not contain any cycle of length large than $2k-2$.

Let $H\in \mathcal{H}$ and denote by $\tilde{H}$ the subgraph of $G$ induced by $\cup_{u\in V(H)}N_U(u)$. Hence, $H$ is the $(k-1)$-core of $\tilde{H}$ and $V(H)=(V(\tilde{H}))^c$.

\begin{lemma}\label{h2}
$\mathcal{H}_2=\emptyset$.
\end{lemma}
\begin{proof}
Suppose to the contrary that $|\mathcal{H}_2|\geqslant 1$. Let $H\in \mathcal{H}_2$. Then $H$ contains a $C_{2k-1}$ as a subgraph and then we assume that $C_{2k-1}=v_1v_2\cdots v_{2k-1}v_1$. We assert that $V(H)=\{v_1,\cdots,v_{2k-1}\}$, otherwise there is a vertex $v\in V(H)\setminus\{v_1,\cdots,v_{2k-1}\}$. Since $H$ is connected, we have $v$ is adjacent to at least one vertex of $\{v_1,\cdots,v_{2k-1}\}$, say $v_1$. Then, we can find that $\{v_2v_3,v_4v_5,\cdots,v_{2k-2}v_{2k-1},vv_1\}$ is a $k$-matching in $H$, which is a contradiction to $G[U]$ is $kK_2$-free. Similarly, we can get $V(\tilde{H})=V(H)=\{v_1,\cdots,v_{2k-1}\}$.

Next, we claim that $U\setminus V(H)$ is an independent set. Otherwise, there is an edge $xy$ in $G[U\setminus V(H)]$, and then we also can find a $k$-matching $\{v_2v_3,v_4v_5,\cdots,v_{2k-2}v_{2k-1},xy\}$ in $G[U]$, which is a contradiction. Hence, we can know that $U\setminus V(H)$ is an independent set. Furthermore, $e(U)= e(H)\leqslant \binom{2k-1}{2}$ and $d_U(u)=0$ for any $u\in U\setminus V(H)$. Therefore
\begin{align*}
\eta_1(U)&=\sum_{u\in U}(d_U(u)-k+1)x_u-e(U)\\
&=\sum_{u\in V(H)}(d_H(u)-k+1)x_u-e(H)+\sum_{u\in U\setminus V(H)}(d_U(u)-k+1)x_u\\
&\leqslant \sum_{u\in V(H)}(d_H(u)-k+1)-e(H)-(k-1)\sum_{u\in U\setminus V(H)}x_u\\
&= e(H)-(k-1)|H|-(k-1)\sum_{u\in U\setminus V(H)}x_u\\
&\leqslant -(k-1)\sum_{u\in U\setminus V(H)}x_u.
\end{align*}
Combining inequality (\ref{e5}), we have $\sum_{u\in U\setminus V(H)}x_u\leqslant\frac{1}{2}k$. Hence, we have
$$\rho=\rho x_{u^*}=\sum_{u\in U}x_u=\sum_{u\in V(H)}x_u+\sum_{u\in U\setminus V(H)}x_u\leqslant |H|+\frac{1}{2}k=\frac{5}{2}k+1.$$
Recall that $m\geqslant 4k^4$. In view of (\ref{e1}), we have $\rho> \sqrt{m}\geqslant 2k^2$. Then we get a contradiction. Hence, $\mathcal{H}_2=\emptyset$.
\end{proof}

\begin{lemma}\label{h3}
$\eta_1(V(H))\leqslant-(k-1)$ for any $H\in \mathcal{H}_3$.
\end{lemma}
\begin{proof}
Recall that $H$ does not contain any cycle of length large than $2k-2$ for any graph $H\in \mathcal{H}_3$. By Lemma \ref{l2}, we have $e(H)\leqslant(k-1)(|H|-1)$. Furthermore,
\begin{align*}
\eta_1(V(H))&=\sum_{u\in V(H)}(d_H(u)-k+1)x_u-e(H)\\
&\leqslant e(H)-(k-1)|H|\\
&\leqslant(k-1)(|H|-1)-(k-1)|H|\\
&=-(k-1).
\end{align*}
\end{proof}

Let $T=\{u\in U\setminus U^c:d_U(u)\leqslant k-2\}$ and $S=U\setminus(U^c\cup T)$. Denote $s=|S|$, $t=|T|$.
\begin{lemma}\label{l5}
$\mathcal{H}_3=\emptyset$ and $\mathcal{H}_1\neq \emptyset$.
\end{lemma}
\begin{proof}
By Lemma \ref{h2}, we have $U^c=\cup_{H\in \mathcal{H}_1\cup \mathcal{H}_3}V(H)$. By Lemmas \ref{vc}, \ref{h1}, and \ref{h3}, we have
$$\eta_1(U)\leqslant\eta_1(U^c)=\sum_{H\in \mathcal{H}_1}\eta_1(V(H))+\sum_{H\in \mathcal{H}_3}\eta_1(V(H))\leqslant-\frac{1}{2}k(k-1)|\mathcal{H}_1|-(k-1)|\mathcal{H}_3|.$$
Now, we will show that $\mathcal{H}_3=\emptyset$. Suppose to the contrary that $\mathcal{H}_3\neq\emptyset$. Then $\mathcal{H}_1=\emptyset$ and $|\mathcal{H}_3|\leqslant\frac{1}{2}k$, otherwise $\eta_1(U)<-\frac{1}{2}k(k-1)$ which is a contradiction to (\ref{e5}). Consequently, we have $U^c=\cup_{H\in \mathcal{H}_3}V(H)$. Recall that  $k\leqslant |H|\leqslant\frac{5k-2}{2}$ for any $H\in \mathcal{H}_3$, we can get $|U^c|\leqslant \frac{5k-2}{2}|\mathcal{H}_3|<\frac{5}{4}k^2$. Since $H$ does not contain any cycle of length large than $2k-2$ for any graph $H\in \mathcal{H}_3$, by Lemma \ref{l2}, we have
\begin{equation}\label{k1}
e(U^c)\leqslant (k-1)|U^c|.
\end{equation}

Recall that $U=U^c\cup S\cup T$. Next, we will show that
\begin{equation}\label{k2}
e(S)+e(S, U^c)\leqslant (k-2)s.
\end{equation}
If $S=\emptyset$, then the inequality holds trivially. Assume that $S\neq\emptyset$. Hence, $T\neq\emptyset$. Otherwise, for any $u\in U$, $d_U(u)\geqslant k-1$ and then $U=U^c$, which is a contradiction. Without loss of generality, assume that $T=\{u_1,\cdots,u_t\}$ and $S=\{u_{t+1},\cdots,u_{t+s}\}$. Since the $k$-core is well-defined, which is not affected by the order of vertex deletion. We can assume that the order in which the vertices are deleted is $u_1,\cdots,u_t,u_{t+1},\cdots,u_{t+s}$. Let $U_1=U$ and $U_i=U_{i-1}\setminus\{u_{i-1}\}$. Then we can get $d_{U_i}(u_i)\leqslant k-2$ for $i\in\{1,\cdots,s+t\}$ for $i\geqslant 2$. One can easily see that $e(S)+e(S, U^c)=\sum_{i=s+1}^{s+t}d_{U_i}(u_i)$ and then $e(S)+e(S, U^c)\leqslant (k-2)s$.

Observe that
\begin{align}\label{k3}
\eta_1(U)&=\sum_{u\in U}(d_U(u)-k+1)x_u-e(U)\notag\\
&=\sum_{u\in T}(d_U(u)-k+1)x_u+\sum_{u\in S\cup U^c}(d_U(u)-k+1)x_u-e(U)\notag\\
&\leqslant -\sum_{u\in T}x_u+\sum_{u\in S\cup U^c}d_U(u)-(k-1)(s+|U^c|)-e(U)\notag\\
&=-\sum_{u\in T}x_u+e(S)+e(S,U^c)+e(U^c)-(k-1)(s+|U^c|).
\end{align}
Combining inequalities (\ref{e5}), (\ref{k1}), (\ref{k2}), and (\ref{k3}), we can obtain
$$-\frac{1}{2}k(k-1)\leqslant \eta_1(U)\leqslant-\sum_{u\in T}x_u-s.$$
Hence, $\sum_{u\in T}x_u\leqslant \frac{1}{2}k(k-1)-s$. Recall that $|U^c|<\frac{5}{4}k^2$. Then
$$\rho=\rho x_u^*=\sum_{u\in U}x_u=\sum_{u\in T}x_u+\sum_{u\in S}x_u+\sum_{u\in U^c}x_u\leqslant \frac{1}{2}k(k-1)-s +s +\frac{5}{4}k^2\leqslant\frac{7}{4}k^2,$$
which is a contradiction to $\rho>\sqrt{m}\geqslant 2k^2$. Hence $\mathcal{H}_3=\emptyset$.

In the following, we will prove $\mathcal{H}_1\neq \emptyset$. Suppose to the contrary that $\mathcal{H}_1= \emptyset$. Then $U^c=\emptyset$ and $U=S\cup T$. Similarly to the proof (\ref{k2}), we can get $e(S)\leqslant (k-2)s$. Hence,
\begin{align*}\label{k3}
\eta_1(U)&=\sum_{u\in U}(d_U(u)-k+1)x_u-e(U)\\
&=\sum_{u\in T}(d_U(u)-k+1)x_u+\sum_{u\in S}(d_U(u)-k+1)x_u-e(U)\\
&\leqslant -\sum_{u\in T}x_u+\sum_{u\in S}d_U(u)-(k-1)s-e(U)\\
&=-\sum_{u\in T}x_u+e(S)-(k-1)s.
\end{align*}
Furthermore, we also can get $\sum_{u\in T}x_u\leqslant \frac{1}{2}k(k-1)-s$. Thus, we similarly get a contradiction. Hence, $\mathcal{H}_1\neq \emptyset$.
\end{proof}

\begin{proof}[\rm{\textbf{Proof of Theorem \ref{th1}}}]
By Lemmas \ref{l1}, \ref{h1}, \ref{h2}, and \ref{l5}, we have
$$\eta_1(U)\leqslant\eta_1(U^c)=\sum_{H\in \mathcal{H}_1}\eta_1(V(H))\leqslant -\frac{1}{2}k(k-1)|\mathcal{H}_1|.$$
Combining inequality (\ref{e5}), we will obtain $\eta_1(V(H))=-\frac{1}{2}k(k-1)$ for $H\in \mathcal{H}_1$ and $|\mathcal{H}_1|=1$.  By using the characterization of equality case in
Lemmas \ref{l2} and \ref{h2} respectively, we get that  $U=U^c=V(H)$, $G[U]\cong S_{|U|,k-1}$, and $x_u=1$ for each $u\in U$ with $d_U(u)\geqslant k$.

Now we may partition $U$ into $U_1$ and $U_2$ where $U_1$ is the vertex set of $K_{k-1}$ in $S_{|U|,k-1}$ and $U_2$ is the independent set in $S_{|U|,k-1}$. Without loss of generality, assume that $U_1=\{u_1,\cdots,u_{k-1}\}$ and $U_1=\{u_k,\cdots,u_{|U|}\}$. Since $|U|>\frac{5k-2}{2}$, we have $|U_2|>\frac{3}{2}k$.

Now it suffices to show that $W=\emptyset$. Suppose to the contrary that there is a vertex $w_0\in W$. Since inequality (\ref{e5}) holds in equality, we also have $e(W)=0$ and $x_w=1$ for each $w\in W$. Then $d_W(w_0)=0$ and $x_{w_0}=1$. Since $N_G(w)\subseteq N_G(u^*)$ and $x_{w_0}=x_{u^*}=1$, we have $N_G(w_0)=N_G(u^*)$. Then we can find a $k$-matching $\{uu_k,u_2u_{k+1},\cdots,u_{k-1}u_{2k-2},w_0u_{2k-1}\}$ in $N_G(u_1)$, and then $G$ contains $F_k$ as a subgraph, which is a contradiction to the assumption of $G$. Hence, we will know $W=\emptyset$. Therefore, $G\cong K_k\vee(\frac{m}{k}-\frac{k-1}{2})K_1$.
\end{proof}

\section{Concluding remarks}
In this paper, we have studied the Brualdi–Hoffman–Tur\'{a}n problem for friendship graph. We denote that $V_{k+1} = K_1\vee P_k$ the fan graph on $k + 1$ vertices. On the other hand, we would like to mention the following conjecture involving fan graph.

\begin{conjecture}\cite{YLP}
Let $k\geqslant 2$ and $m$ large enough. If $G$ is $V_{2k+1}$-free or $V_{2k+2}$-free with given size $m$, then
$$\rho(G)\leqslant \frac{k-1+\sqrt{4m-k^2-1}}{2}$$
with equality if and only if $G\cong K_k\vee(\frac{m}{k}-\frac{k-1}{2})K_1$.
\end{conjecture}
In \cite{YLP} and \cite{ZW}, they both proved the conjecture holds for $V_5$. In \cite{ZW2}, they confirmed the conjecture holds for $V_7$. Observe that $F_k$ is a subgraph of $V_{2k+1}$. If the conjecture holds, then we also can get the Brualdi–Hoffman–Tur\'{a}n problem for friendship graph.

\end{document}